\documentclass[letterpaper, 10 pt, conference]{ieeeconf}  

\IEEEoverridecommandlockouts                              
\usepackage{booktabs}
\usepackage{graphicx}
\usepackage{verbatim}
\usepackage{xcolor}
\usepackage{extarrows}
\usepackage{comment}
\usepackage{cases}
\usepackage{subfig}

\usepackage{amsmath,amsfonts,amsthm} 

\newtheorem{theorem}{Theorem}
\newtheorem{proposition}{Proposition}
\newtheorem{lemma}{Lemma}

\newtheorem{remark}{Remark}

\title{\LARGE \bf
A Continuous Threshold Model of Cascade Dynamics
}

\author{Yaofeng Desmond Zhong and Naomi Ehrich Leonard
\thanks{This research was supported by ARO grant W911NF-18-1-0325 and ONR grants N00014-14-1-0635, N00014-19-1-2556. The authors are with the Department of Mechanical and Aerospace Engineering, Princeton University, Princeton, NJ 08544, USA. {\tt\small \{y.zhong,naomi\}@princeton.edu}}%
}

\begin{document}

\maketitle
\thispagestyle{empty}
\pagestyle{empty}

\begin{abstract}

We present a continuous threshold model (CTM) of cascade dynamics for a network of agents with real-valued activity levels that change continuously in time. The model generalizes the linear threshold model (LTM) from the literature, where an agent becomes active (adopts an innovation) if the fraction of its neighbors that are active is above a threshold. With the CTM we study the influence on cascades of heterogeneity in thresholds for a network comprised of a chain of three clusters of agents, each distinguished by a different threshold. The system is most sensitive to change as the dynamics pass through a bifurcation point: if the bifurcation is supercritical the response will be contained, while if the bifurcation is subcritical the response will be a cascade. We show that there is a subcritical bifurcation, thus a cascade, in response to an innovation if there is a large enough disparity between the thresholds of sufficiently large clusters on either end of the chain; otherwise the response will be contained. 
\end{abstract}

\section{Introduction}
\label{intro}

Cascade dynamics refer to the spread of an activity or innovation among a group of agents. They have been modelled as discrete-time, discrete-valued state dynamics in which an agent accepts or rejects an innovation at each time step after comparing the fraction of its neighbors who have accepted the innovation to a threshold between 0 and 1. 
This model is referred to as the linear threshold model (LTM). 

The LTM was first introduced in \cite{granovetter1978threshold,schelling1978micromotives}. Kempe et al.\ \cite{kempe2003maximizing} and Lim et al.\ \cite{lim2015simple} studied the LTM with uniformly drawn thresholds. Zhong et al.\ \cite{zhong2017linear} generalized the LTM to duplex networks where there exist two different types of interactions among the agents. All of these results leverage thresholds drawn from a uniform distribution. Acemoglu et al.\ \cite{acemoglu2011diffusion} studied cascade dynamics using the LTM with deterministic thresholds. In this case, the analysis becomes challenging. Yang et al. \cite{yang_minimizing_2017} studied the influence minimization problem for the deterministic LTM by formulating the problem as a linear integer programming problem. Fardad and Kearney \cite{fardad2017linear} studied the optimal seeding problem of cascade failure using a relaxation of the deterministic LTM and formulated the problem as a convex program. Pinheiro et al. \cite{2018arXiv181108525P} introduced nonlinearity to the fraction of neighbors and investigated cascades on the all-to-all network. However, the analysis of cascades remains challenging for more general network graphs, heterogeneous agents, and deterministic thresholds.

We propose a continuous threshold model (CTM) that introduces continuous-time, real-valued state dynamics with nonlinearity. The CTM is adapted from nonlinear consensus dynamics \cite{gray2018multiagent}, which exhibit very rapid transitions in system state when the system passes through a bifurcation point. Using the proposed model, we can thus prove conditions for sudden cascades using methods from nonlinear dynamics.

We show how the CTM generalizes the LTM. We then use the CTM to study the influence on cascades of heterogeneity in (deterministic) thresholds among the agents. We consider networks comprised of three clusters of agents, each cluster associated with a different threshold and thus a different level of responsiveness to the state of neighbors. A lower threshold implies a higher responsiveness. Let $\epsilon > 0$. Cluster 1 is the ``high response'' cluster where agents use a threshold $\mu = 1/2 - \epsilon$.  Cluster 2 is the ``low response'' cluster where agents use a threshold $\mu = 1/2 + \epsilon$.  Cluster 3 is the ``neutral response'' cluster where agents use a threshold $\mu = 1/2$.


For the networks considered, we show how the size $n$ of clusters 1 and 2 and the strength of the disparity $2\epsilon$ between their thresholds  determines whether or not there is a cascade in response to the introduction of an innovation. When the size and disparity are small the group exhibits a contained response, whereas when the size and disparity are sufficiently large the group exhibits a rapid increase in state, indicating a cascade.  This is an interesting and even surprising result.


In the analysis, the contained response corresponds to a supercritical pitchfork bifurcation in the dynamics and the cascade corresponds to a subcritical pitchfork bifurcation. 
Gray. et al. \cite{gray2018multiagent} observed the transition from supercritical to subcritical pitchfork in nonlinear consensus dynamics;
however, they did not prove conditions under which this transition exists. The transition is also exhibited in the replicator-mutator dynamics studied by Dey et al. \cite{dey2018feedback}. They showed the existence of the transition for the system with two strategies.
We derive rigorous conditions for existence of the transition in the CTM dynamics of a network of $N$ agents comprised of three clusters distinguished by their thresholds. 

Our contribution is twofold. First, we present a new model of cascade dynamics that generalizes the LTM. Second, we provide new results on how cascades are influenced by heterogeneity in thresholds of agents, i.e., how ready an agent is to change its behavior in response to its neighbors.  For a network of three clusters, we derive a necessary condition for there to be a cascade that depends on the sizes of the clusters with disparate thresholds and $N$.  We show that there is a corresponding critical value $\epsilon^* > 0$ such that we can expect a cascade when $\epsilon > \epsilon^*$ but not when $\epsilon < \epsilon^*$. 

Section \ref{CTM} describes the CTM dynamics 
and equivalence to the LTM. In Section \ref{three_clusters}, we specialize the dynamics to a family of network graphs with three clusters. In Section \ref{contidion}, we define a cascade for the CTM and prove conditions under which a cascade occurs. Section \ref{example} provides an example. 
\section{Continuous Threshold Model}
\label{CTM}
The proposed model generalizes the discrete linear threshold model (LTM), and it is inspired by the Hopfield network dynamics \cite{hopfield1982neural} and adapted from nonlinear consensus dynamics \cite{gray2018multiagent}. To describe complex contagion within a group of $N$ agents, let $x_i  \in \mathbb{R}$ be the state of agent $i$, representing the activity level of agent $i$. Agent $i$ is said to be active (inactive) if $x_i>0$ $(x_i<0)$. A greater absolute value of the state $|x_i|$ means that agent $i$ is more active (more inactive).

Interactions among agents, i.e., who can sense or communicate with whom, are encoded in graph $G=(V, E)$, where $V=\{1, \ldots, N\}$ is the set of $N$ agents and $E \subset V\times V$ is the edge set representing interactions. An edge $e_{ij} \in E$ implies that $j$ is a neighbor of $i$. We assume there are no self-loops, i.e., $e_{ii} \notin E$. The graph adjacency matrix $A \in \mathbb{R}^{N \times N}$ is a matrix with elements of $0$ and $1$, where $a_{ij} = 1$ if and only if $e_{ij} \in E$. The degree matrix $D \in \mathbb{R}^{N\times N}$ is a diagonal matrix with diagonal entries $d_i = \sum_{j=1}^N a_{ij}$.

The {\em continuous threshold model} (CTM) defines the change in activity level of each agent over time as a function of the agent's current state and the state of its neighbors:
\begin{equation}
    \label{dynamics}
    \dot  x_i = - d_i x_i + \sum_{j=1}^N a_{ij} uS(vx_j) + d_i (1-2\mu_i).
\end{equation}
$\mu_i$ can be interpreted as the threshold of agent $i$ and $u,v>0$ 
as control parameters. $S: \mathbb{R} \rightarrow [-1,1]$ is a smooth, odd sigmoidal function that satisfies the following conditions: $S'(x)>0, \forall x \in \mathbb{R}$; $S'(0)=1$; and $\mathrm{sgn}(S''(x)) = -\mathrm{sgn}(x)$, where $(\cdot)'$ denotes the derivative and $\mathrm{sgn}$ is the sign function. 
Sigmoids are ubiquitous in models of biological and physical systems; here they serve to saturate influence from neighbors.

The control parameter $u$ can be interpreted as the strength of the ``social sensitivity'' since a larger $u$ means a greater attention to social cues. The control parameter $v$ can be interpreted as the strength of the ``social effort'' since a larger $v$ means a stronger signal sent by neighbors. The CTM generalizes the LTM in the following way. Let $u=1$ and $v\rightarrow +\infty$, then the dynamics (\ref{dynamics}) become 
\begin{equation}
    \label{dynamics_LTM}
    \frac{1}{d_i}\dot x_i = 
    -  x_i + 2 \Big(\sum_{j=1}^N \frac{a_{ij}}{d_i} \frac{S(vx_j)+1}{2}  -\mu_i \Big).
\end{equation}
Since $v\rightarrow +\infty$, the sigmoidal function approaches the sign function, which maps a negative state to -1 and a positive state to +1. The fraction $(S(vx_j)+1)/2$ maps a negative state to 0 and a positive state to 1. So the summation gives the fraction of active neighbors. Thus, the difference between the summation and $\mu_i$ is the comparison of fraction of active neighbors of agent $i$ to the threshold of agent $i$. 

Now consider the equilibrium points of (\ref{dynamics_LTM}). If agent $i$'s fraction of active neighbors is greater than its threshold, the steady-state value of $x_i$ is positive and agent $i$ is active; otherwise, the steady-state value of $x_i$ is negative and agent $i$ is inactive. Let unseeded agents be defined by a negative initial state and seeded agents by a large positive initial state. 
Then it follows that the dynamics (\ref{dynamics_LTM}) behave like the LTM with deterministic thresholds $\mu_i$, $i = 1, \dots, N$ in the sense that the ordering of unseeded agents switching from inactive to active is the same and hence the steady states are the same. 

We study the CTM for $v=1$ and $u$ a feedback control that depends on the {\em slow} filtered average state $\bar x_s$.  We let 
$u = u_0 S(\kappa|\bar x_s|)$ and
$\dot{\bar{x}}_s = \kappa_s \left( \bar x - \bar x_s\right)$
with $\bar x = \sum_{i=1}^N x_i/N$ and $u_0,\kappa,\kappa_s >0$.

\section{Networks with Three Clusters}
\label{three_clusters}
Consider a class of networks with $N$ agents in three clusters as in Section~\ref{intro}:  every agent $i$ in the high response cluster 1 has $\mu_i = 1/2 - \epsilon$, every agent $j$ in the low response cluster 2 has $\mu_j = 1/2 + \epsilon$ and every agent $k$ in the neutral response cluster 3 has $\mu_k = 1/2$. 
Let there be $n$ agents in cluster 1, $n$ in cluster 2, and $N-2n$ in cluster 3.

Let all edges in the network be undirected, i.e., $a_{ij} = a_{ji}$. Within each cluster, the graph is all-to-all. Each agent in cluster 3 is connected to each agent in clusters 1 and 2, and there are no connections between agents in cluster 1 and agents in cluster 2. The network is motivated by environments in which a feature that influences threshold adoption is distributed and agents with biased thresholds  interact with agents with similarly biased or unbiased thresholds. One example is 
a population clustered by age bracket, where the youngest are most likely and the eldest least likely to purchase a new technology when their friends do. 
Another is a spatially distributed group in which agents at one end measure smoke and adopt a low threshold to flee a fire, and agents on the other end miss the smoke and adopt a high threshold. 
Fig.~\ref{graph_11nodes} shows  a network with $N=11$ and $n=4$. 
\begin{figure}[h]
    \centering
    \includegraphics[width=0.4\textwidth]{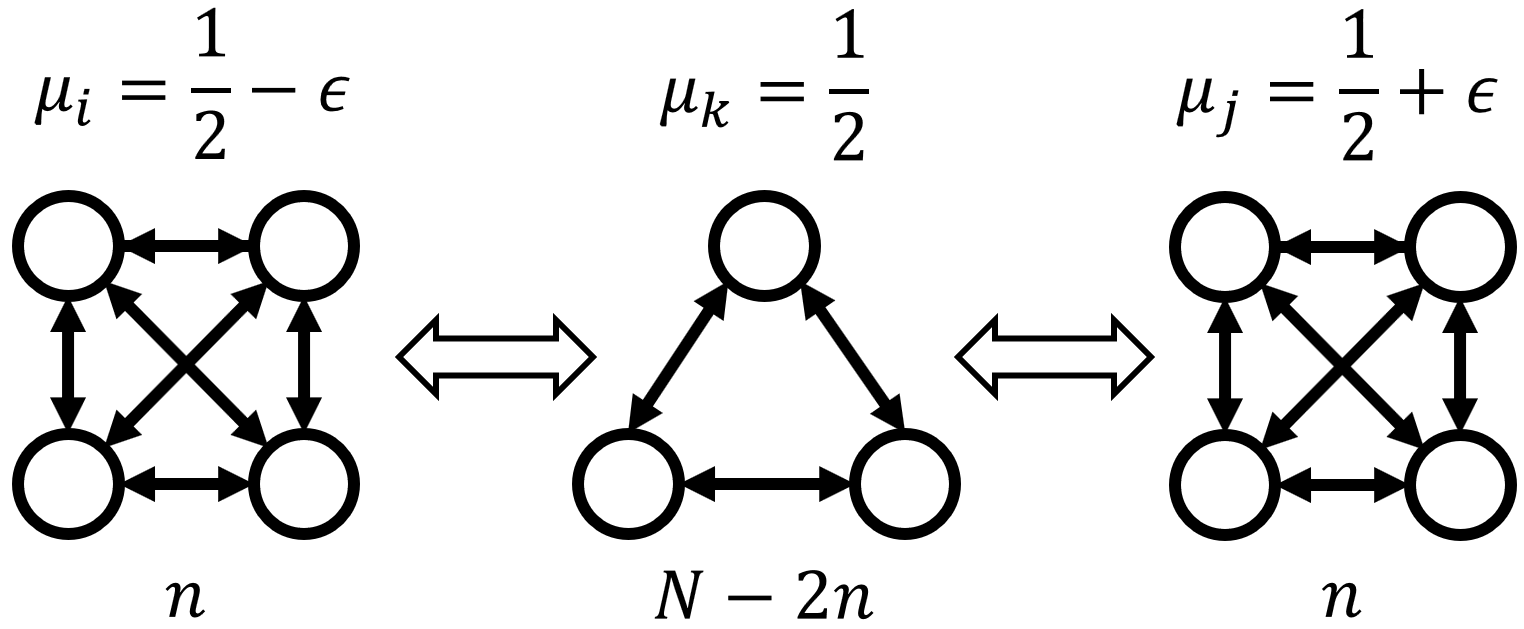}
    \caption{Network with three clusters: $N=11$ and $n=4$. Cluster 1 (high response) is on the left,  cluster 2 (low response) is on the right, and cluster 3 (neutral response) is in the middle. White arrows indicate all-to-all, undirected connections between nodes in clusters.}
        \label{graph_11nodes}
\end{figure}


With an approach similar to that in  Theorem~4 of \cite{gray2018multiagent}, it can be shown that the trajectories of  (\ref{dynamics}) converge exponentially to the three-dimensional manifold where all the states in the same cluster are the same. Let $y_k$ be the average state of cluster $k = 1, 2, 3$. The reduced dynamics are
\begin{align}
    \dot{y}_1 = &-(N-n-1)y_1 +(n-1)uS(y_1)  \nonumber\\
                &+ (N-2n)uS(y_3)+2(N-n-1)\epsilon \label{reduce_3_1}\\
    \dot{y}_2 = &-(N-n-1)y_2 +(n-1)uS(y_2) \nonumber\\
                &+ (N-2n)uS(y_3)-2(N-n-1)\epsilon \label{reduce_3_2}\\
    \dot{y}_3 = &-(N-1)y_3 + (N-2n-1)uS(y_3) \nonumber\\
                &+ nuS(y_1) + nuS(y_2) \label{reduce_3_3}.
\end{align}
Let $\mathbf{y} = [y_1,y_2,y_3]^T$ 
and $F(\mathbf{y},u,\epsilon)$ the RHS of (\ref{reduce_3_1})-(\ref{reduce_3_3}). 

$F$ commutes with the action of the nontrivial element of $Z_2$, the cyclic group of order 2, represented by the matrix: 
$\gamma = 
    \begin{bmatrix}
    0 & -1 & 0 \\
    -1 & 0 & 0 \\
    0 & 0 & -1
    \end{bmatrix}$,
i.e., $F(\gamma \mathbf{y},u,\epsilon)= \gamma F(\mathbf{y},u,\epsilon)$. 
This 
implies a $Z_2$-symmetric singularity. We show in the next section that $F$ 
possesses a pitchfork bifurcation, and we prove a necessary condition for 
the transition from a supercritical to a subcritical pitchfork.

The bifurcations are illustrated in Fig.~\ref{fig:symmetric} where  
the horizontal axis represents the bifurcation parameter $u$ and the vertical axis the average state $\bar{y} = (ny_1 + ny_2 + (N-2n)y_3)/N$. Blue curves represent stable solutions and red curves unstable solutions to (\ref{reduce_3_1})-(\ref{reduce_3_3}). The neutrally active average state $\bar{y} = 0$ is always a solution, and it is stable for $u < u^c$ and unstable for $u > u^c$, where $u=u^c$ is the bifurcation point. 


Due to the feedback, the social sensitivity parameter $u$ 
will slowly increase when an innovation has been introduced 
and cross the bifurcation point where the system is highly sensitive to change (see \cite{gray2018multiagent} for generalizations to heterogeneous $u_i$). For initial conditions corresponding to one or more active agents such that $\bar{y}(0)>0$, the solution will increase as shown by the green curves in Fig.~\ref{fig:symmetric}. 
The trajectory in the supercritical pitchfork slowly follows the positive branch of the pitchfork as $u$ is increased just above the critical value $u^c_{\mathrm{sup}}$. We define this slow increase in $\bar y$ as a {\em contained response}. The trajectory in the supercritical pitchfork jumps up to the positive branch as $u$ is increased just above the critical value $u^c_{\mathrm{sub}}$.  We define the jump as a {\em cascade} since it implies a rapid spread of the innovation.

\begin{figure}%
    \centering
    \subfloat{{\includegraphics[width=4.0cm]{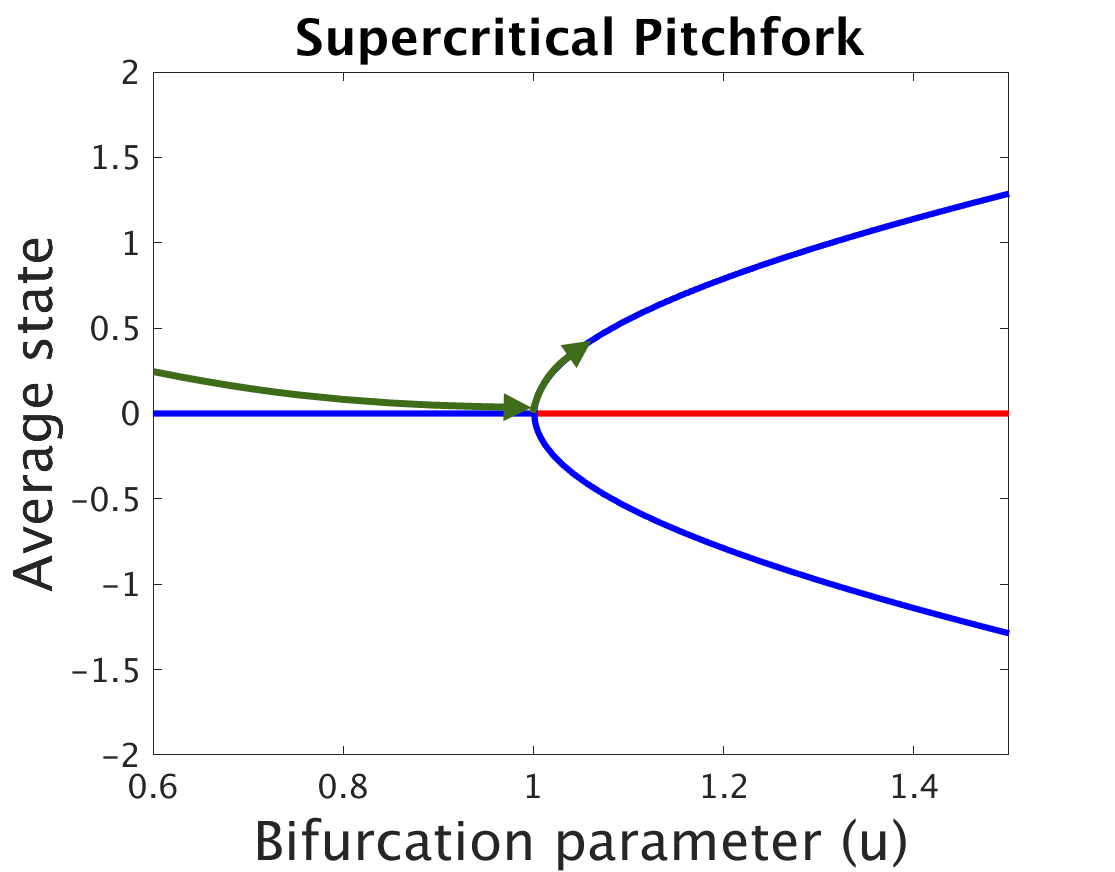} }}%
    \subfloat{{\includegraphics[width=4.4cm]{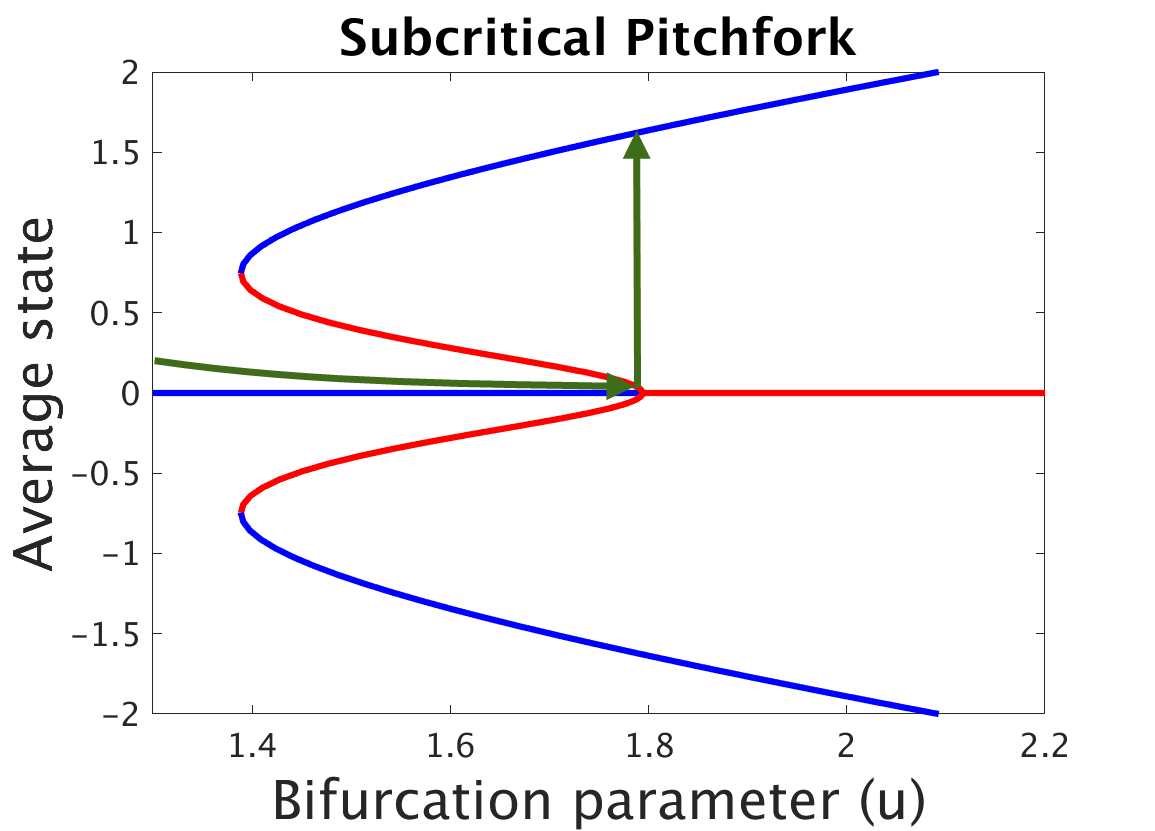} }}%
    \caption{Pitchfork bifurcation diagrams: supercritical (left) and subcritical (right). Blue (red) curves are stable (unstable) solutions. 
    Green curves are trajectories 
    as $u$ slowly increases.}%
    \label{fig:symmetric}%
\end{figure}

\begin{figure}%
    \centering
    \subfloat{{\includegraphics[width=4.2cm]{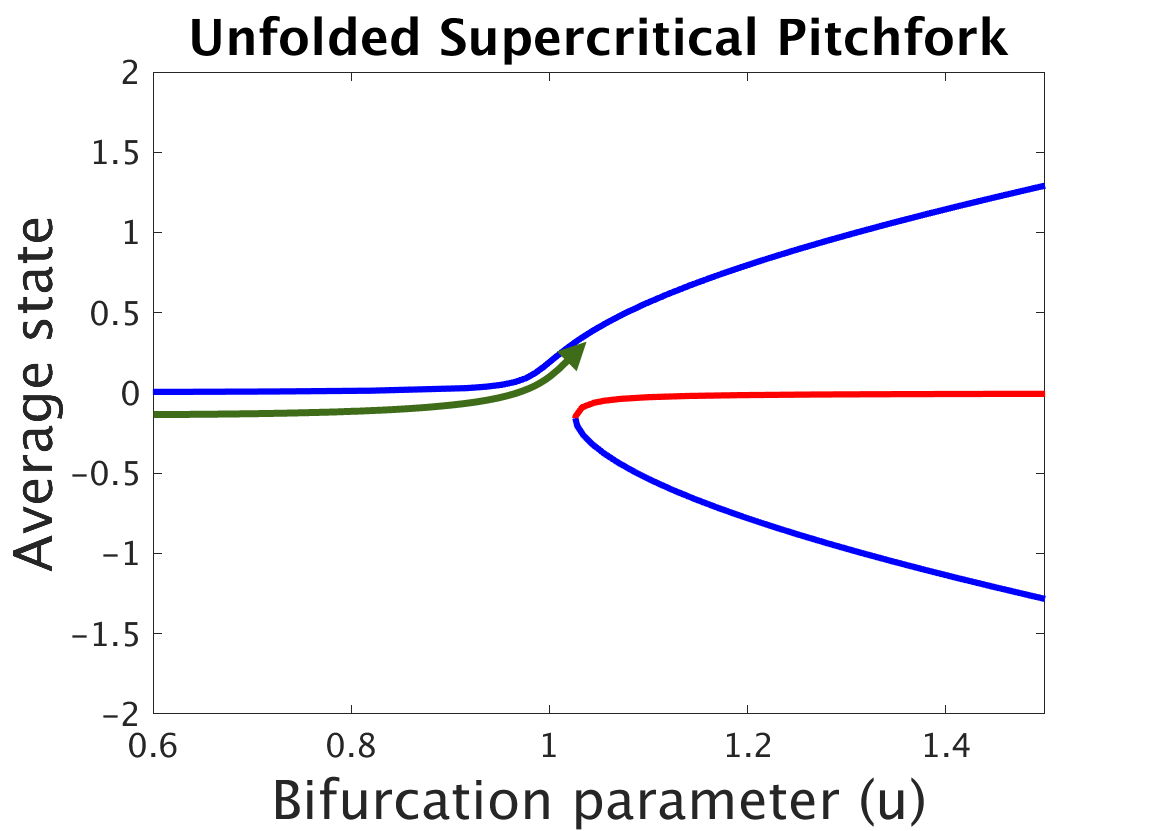} }}%
    \subfloat{{\includegraphics[width=4.2cm]{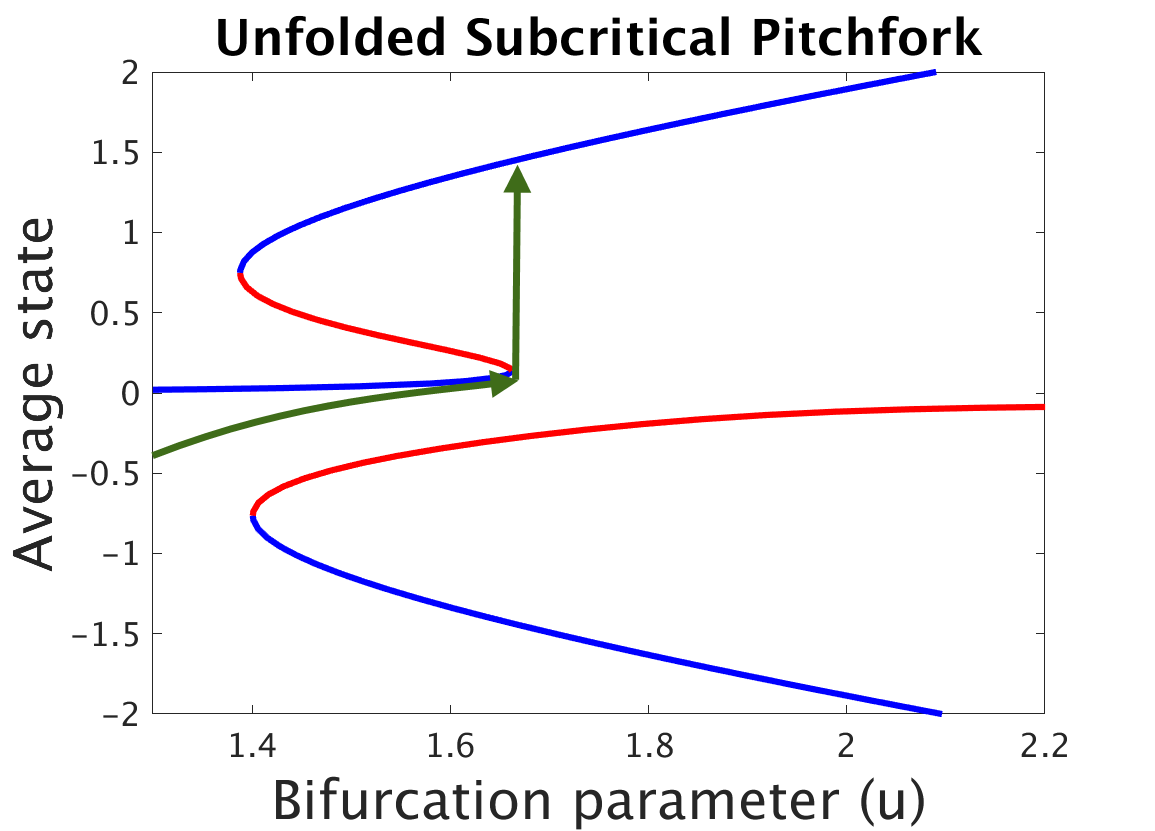} }}%
    \caption{Unfolded pitchfork bifurcation diagrams: supercritical (left) and subcritical (right). Colors are as in Fig.~\ref{fig:symmetric}.} 
    \label{fig:perturbed}%
\end{figure}

Fig.~\ref{fig:perturbed} shows what happens to the bifurcation diagrams in the presence of a small positive input to the dynamics (\ref{reduce_3_1}), corresponding to the introduction of an innovation as an external cue rather than as seeded positive initial conditions. The resulting ``unfolded'' supercritical pitchfork still exhibits the contained response and the ``unfolded'' subcritical pitchfork still exhibits the cascade.  
Here, even with an initial condition corresponding to an average initial state $\bar y(0)<0$, the innovation can still trigger a cascade. 
\section{Conditions for Cascade}
\label{contidion}
Although the transition from supercritical pitchfork to subcritical pitchfork has been observed in \cite{gray2018multiagent}, it is unclear for what parameter values 
the transition exists in the CTM. In this section, we first show conditions for existence of a pitchfork bifurcation and then for existence of the transition. 

By $Z_2$-symmetry,  $\mathbf{y}^{\star} = [y^{\star},-y^{\star},0]^T$ is always an equilibrium of (\ref{reduce_3_1})-(\ref{reduce_3_3}), where for a given $u$ and $\epsilon$, $y^{\star}$ satisfies
\begin{align}
    \label{y_star_eqn}
    -(N-n-1)y^{\star}+(n-1)uS(y^{\star})+2(N-n-1)\epsilon = 0.
\end{align}
Consider a perturbation to the trivial solution $\mathbf{y}^{\star}$ and denote the perturbed solution as $\mathbf{y}^{\star} + \Delta \mathbf{y}= [y^{\star} + \Delta y_1,-y^{\star} + \Delta y_2, \Delta y_3]^T$. We ask the question, could there be a nontrivial equilibrium point where $\Delta \mathbf{y} \neq \mathbf{0}$? The change from no nontrivial equilibria to the existence of nontrivial equilibria corresponds to the bifurcation point. The following perturbation analysis allows us to reduce the dynamics near the trivial equilibrium to one-dimensional dynamics that match the normal form of a pitchfork bifurcation. We can then evaluate if the pitchfork bifurcation is supercritical or subcritical by examining the sign of coefficients in the reduction.

We use the  Taylor series expansion to  third order: 
\begin{align}
    S(y^{\star}\!+\!\Delta y_1) =& S(y^{\star}) + S'(y^{\star}) \Delta y_1 + \frac{1}{2}S''(y^{\star}) (\Delta y_1)^2 \nonumber \\
        &  + \frac{1}{6}S'''(y^{\star}) (\Delta y_1)^3 + o((\Delta y_1)^3) \label{delta_1}\\
    S(-y^{\star}\!+\!\Delta y_2) \!=\! &S(-y^{\star}) \!+\! S'(-y^{\star}) \Delta y_2 \! + \! \frac{1}{2}S''(-y^{\star}) (\Delta y_2)^2 \nonumber \\
        &  + \frac{1}{6}S'''(-y^{\star}) (\Delta y_2)^3 + o((\Delta y_2)^3) \label{delta_2}  \\
    S(\Delta y_3) =& S'(0) \Delta y_3  +  \frac{1}{2}S''(0) (\Delta y_3)^2 \nonumber \\
    &+ \frac{1}{6}S'''(0) (\Delta y_3)^3  + o((\Delta y_3)^3)\label{delta_3}.
\end{align}
Since the sigmoidal is an odd function, we have that $S(y^{\star}) = -S(-y^{\star})$, 
 $S'(y^{\star}) =S'(-y^{\star})$,  $S''(y^{\star}) = - S''(-y^{\star})$ and $S'''(y^{\star}) =  S'''(-y^{\star})$. Without loss of generality, we use the hyperbolic tangent as the sigmoidal function from now on. For other types of sigmoidal functions, the analysis follows similarly. Now Eqn. (\ref{delta_3}) has the form of 
\begin{align}
    \mathrm{tanh}(\Delta y_3) = \Delta y_3 -\frac{1}{3} \Delta y_3^3  + o((\Delta y_3)^3).
    \label{tanh_delta_3}
\end{align}

Lemma~\ref{lemma_cubic} and Proposition~\ref{prop_reduce_1} will be used to derive the  conditions for transition from supercritical to subcritical bifurcation. 

\begin{lemma}
    \label{lemma_cubic}
    Assume $\Delta x\in \mathbb{R}$ and $\Delta y \in \mathbb{R}$ have small magnitudes and satisfy 
    \begin{equation}
         a (\Delta x)^3 + b (\Delta x)^2 + c \Delta x = - \Delta y + \frac{1}{3} (\Delta y^3)  + o((\Delta y)^3). \label{cubic}
    \end{equation}
    Then we have 
    \begin{equation*}
        \Delta x \!=\! -\frac{1}{c} \Delta y \!-\! \frac{b}{c^3} (\Delta y)^2\!+\! \Big(\frac{1}{3c}\!-\!\frac{2b^2}{c^5}\!+\!\frac{a}{c^4}\Big) (\Delta y)^3 \!+\! o((\Delta y^3)).
    \end{equation*}
\end{lemma}

\begin{proof}
    Assume $\Delta x = \alpha_1 \Delta y +\alpha_2 (\Delta y)^2+ \alpha_3 (\Delta y)^3 + o((\Delta y)^3)$, then we have 
    \begin{align*}
        (\Delta x)^2 &= \alpha_1^2 (\Delta y)^2 + 2\alpha_1 \alpha_2 (\Delta y)^3 + o((\Delta y)^3) \\
        (\Delta x)^3 &= \alpha_1^3 (\Delta y)^3 + o((\Delta y)^3).
    \end{align*}
    Substitute the above equations into Eqn. (\ref{cubic}) and equate the coefficients in front of $\Delta y$, $(\Delta y)^2$, $(\Delta y)^3$ in the LHS and RHS, respectively. We get three equations. Then we solve for $\alpha_1$, $\alpha_2$ and $\alpha_3$ and get the result.
\end{proof}

\begin{proposition}
\label{prop_reduce_1}
    Consider dynamics (\ref{reduce_3_1})-(\ref{reduce_3_3}) with $S(\cdot) = \mathrm{tanh}(\cdot)$. The conditions for equilibria of the perturbed dynamics of (\ref{reduce_3_1})-(\ref{reduce_3_3}) around $y^{\star}$ can be reduced to the following single condition:
    \begin{align}
        \label{pitchfork_y3}
        \frac{\mathrm{d}}{\mathrm{d}t} \big(\Delta y_3\big) = \lambda_1 \Delta y_3 + \lambda_3 (\Delta y_3)^3 +o((\Delta y_3)^3)= 0,
    \end{align}
    where $\lambda_1$ and $\lambda_3$ depend on $y^{\star}$, $u$, $N$, $n$ as follows:
    \begin{align}
    \lambda_1(y^{\star}, u, N, n) = &-(N-1) - \Big(1+\frac{n+1}{n-1}(N-2n)\Big)u \nonumber \\
        &- \frac{n(N-n-1)}{n-1}\frac{2}{c} \label{lambda1_origin}\\
    \lambda_3(y^{\star}, u, N, n) = &\frac{1}{3}\Big(1+\frac{n+1}{n-1}(N-2n)\Big)u \nonumber \\
        &+ \frac{n(N-n-1)}{n-1} \Big(\frac{2}{3c}-\frac{4b^2}{c^5}+\frac{2a}{c^4}\Big) \label{lambda3_origin}
\end{align}
with
\begin{align}
    a(y^{\star}, N, n) &= \frac{n-1}{N-2n} \frac{1}{6} \mathrm{tanh}'''(y^{\star}) \label{a}\\
    b(y^{\star}, N, n) &= \frac{n-1}{N-2n} \frac{1}{2} \mathrm{tanh}''(y^{\star}) \label{b}\\
    c(y^{\star}, u, N, n) &= \frac{n-1}{N-2n} \mathrm{tanh}'(y^{\star}) - \frac{N-n-1}{(N-2n)u} \label{c}.
\end{align}
\end{proposition}

\begin{proof}
First we substitute $\mathbf{y}^{\star} + \Delta \mathbf{y}$ into the RHS of  (\ref{reduce_3_1}), (\ref{reduce_3_2}) and set them equal to zero.
With  (\ref{delta_1}), (\ref{delta_2}) and (\ref{tanh_delta_3}), we get
\begin{align}
    -&(N -n-1)\Delta y_1 +(n-1)u\Big(S'(y^{\star}) \Delta y_1 \nonumber \\
    &+ \frac{1}{2}S''(y^{\star}) (\Delta y_1)^2 + \frac{1}{6}S'''(y^{\star}) (\Delta y_1)^3\Big) \nonumber \\
    &+(N-2n)u(\Delta y_3 -\frac{1}{3} (\Delta  y_3)^3) +o((\Delta y_3)^3)=0  \label{delta_y1}\\
    -&(N -n-1)\Delta y_2 +(n-1)u\Big(S'(-y^{\star}) \Delta y_2 \nonumber \\
    &+ \frac{1}{2}S''(-y^{\star}) (\Delta y_2)^2 + \frac{1}{6}S'''(-y^{\star}) (\Delta y_2)^3\Big)\nonumber \\
    &+(N-2n)u(\Delta y_3 -\frac{1}{3} (\Delta  y_3)^3) +o((\Delta y_3)^3)=0 \label{delta_y2}.
\end{align}
Eqns. (\ref{delta_y1}) and (\ref{delta_y2}) can be written as follows: 
\begin{align*}
    a (\Delta y_1)^3 \!+\! b (\Delta y_1)^2 \!+ \!c \Delta y_1 &= - \Delta y_3 \!+\! \frac{1}{3} (\Delta y_3)^3\!+ \!o((\Delta y_3)^3) \\
    a (\Delta y_2)^3 \!- \!b (\Delta y_2)^2 \!+ \!c \Delta y_2 &= - \Delta y_3 \!+ \!\frac{1}{3} (\Delta y_3)^3\! + \!o((\Delta y_3)^3) 
\end{align*}
with $a$, $b$, and $c$ given by (\ref{a}), (\ref{b}), and (\ref{c}), respectively. 


By Lemma \ref{lemma_cubic}, we have 
\begin{align}
    \Delta y_1 =& -\frac{1}{c} \Delta y_3 - \frac{b}{c^3} (\Delta y_3)^2+ \Big(\frac{1}{3c}-\frac{2b^2}{c^5}+\frac{a}{c^4}\Big) (\Delta y_3)^3 \nonumber \\
    &+ o((\Delta y_3)^3) \label{delta_y1_cubic}\\
    \Delta y_2 =& -\frac{1}{c} \Delta y_3 + \frac{b}{c^3} (\Delta y_3)^2+ \Big(\frac{1}{3c}-\frac{2b^2}{c^5}+\frac{a}{c^4}\Big) (\Delta y_3)^3 \nonumber \\
    &+ o((\Delta y_3)^3) \label{delta_y2_cubic}.
\end{align}


We substitute $\mathbf{y}^{\star} + \Delta \mathbf{y}$ into the RHS of  (\ref{reduce_3_3}) and set it equal to zero.  We leverage  (\ref{y_star_eqn}), (\ref{delta_y1}) and (\ref{delta_y2}) to get
\begin{align}
    \frac{\mathrm{d}\Delta y_3}{\mathrm{d}t} =& -(N-1)\Delta y_3 + (N-2n-1)u(\Delta y_3 -\frac{1}{3}\Delta y_3^3) \nonumber\\
    &+\frac{n(N-n-1)}{n-1} (\Delta y_1 + \Delta y_2) + o((\Delta y_3)^3)= 0.
    \label{eq:deltay3rate}
\end{align}
As we are able to express $\Delta y_1$ and $\Delta y_2$ in terms of $\Delta y_3$ from (\ref{delta_y1_cubic}) and (\ref{delta_y2_cubic}), we can substitute them into (\ref{eq:deltay3rate}). 
This gives a reduction of the conditions for equilibria of (\ref{reduce_3_1})-(\ref{reduce_3_3}) from three equations to a single equation in terms of $\Delta y_3$. We can see clearly the terms with $(\Delta y_3)^2$ cancel out, which is consistent with the $Z_2$-symmetry. We then get our main equation (\ref{pitchfork_y3}) with $\lambda_1$ and $\lambda_3$ given by (\ref{lambda1_origin}) and (\ref{lambda3_origin}), respectively.


\end{proof}

We examine  (\ref{pitchfork_y3}) from Proposition \ref{prop_reduce_1}. If $\lambda_3<0$, and $\lambda_1$ crosses zero from negative to positive, $\Delta y_3$ undergoes a supercritical pitchfork bifurcation. For $\lambda_1<0$ and $|\lambda_1|$ sufficiently small, there is a single stable solution at $\Delta y_3 = 0$, which implies $\Delta y_1 = \Delta y_2 = 0$. In this case $\mathbf{y}^{\star}$ is a stable equilibrium of (\ref{reduce_3_1})-(\ref{reduce_3_3}) and there are no other solutions nearby. For $\lambda_1>0$ and $|\lambda_1|$ sufficiently small,  $\Delta y_3 = 0$ is unstable and two stable equilibria $\Delta y_3 = \pm \sqrt{-\lambda_1/\lambda_3}$ appear. 

If $\lambda_3>0$, and $\lambda_1$ crosses zero from negative to positive, $\Delta y_3$ undergoes a subcritical pitchfork bifurcation. For $\lambda_1<0$ and $|\lambda_1|$ sufficiently small, there are two unstable equilibria $\Delta y_3 = \pm \sqrt{-\lambda_1/\lambda_3}$ and one stable equilibrium  $\Delta y_3=0$. For $\lambda_1>0$ and $|\lambda_1|$ sufficiently small, the three equilibria collapse into one unstable equilibrium $\Delta y_3=0$. 

The following proposition gives the condition for existence of the transition from supercritical to subcritical pitchfork. 
\begin{proposition}
    \label{prop_tran}
    The transition from a  supercritical pitchfork bifurcation to a  subcritical pitchfork bifurcation of dynamics (\ref{reduce_3_1})-(\ref{reduce_3_3}) with $S(\cdot) = \mathrm{tanh}(\cdot)$ occurs when $\lambda_3$ crosses zero from negative to positive. 
    The condition for the transition is 
    \begin{align}
        \lambda_3(y^{\star},N,n) = 0,
        \label{transition}
    \end{align}
    where
    \begin{align}
        \label{lambda3}
        \lambda_3&(y^{\star},N,n) = -\frac{1}{3}(N-1) + \frac{n(N-n-1)}{n-1} \nonumber \\
        &\times \frac{2a(y^{\star},N,n)c(y^{\star},u(y^{\star},N,n),N,n)-4b^2(y^{\star},N,n)}{c^5(y^{\star},u(y^{\star},N,n),N,n)},
    \end{align}
    and
    \begin{align}
    \label{u_y_star}
    u(y^{\star}, N,n) = \frac{-c_1 + \sqrt{c_1^2-4c_2c_0}}{2c_2} = \frac{-2c_0}{\sqrt{c_1^2-4c_2c_0}+c_1}.
\end{align}
Here, $a$, $b$, $c$ are  given by  (\ref{a}), (\ref{b}), (\ref{c}) and 
    \begin{align}
    c_2(y^{\star},N,n) =& \Big(n+1+\frac{n-1}{N-2n} \Big) \mathrm{tanh}'(y^{\star}) \label{c2} \\
    c_1(y^{\star},N,n) =& \frac{(N-2n-1)(N-n-1)}{(N-2n)}  \nonumber\\
    &+ \frac{(N-1)(n-1)}{N-2n} \mathrm{tanh}'(y^{\star}) \label{c1} \\
    c_0(y^{\star},N,n) =& - \frac{(N-n-1)(N-1)}{N-2n}
    \label{c0}.
\end{align}
The value of $y^{\star}$ at the transition is the solution of (\ref{transition}). The value of $u$ at the transition is a function of $y^{\star}$, $N$, and $n$  (\ref{u_y_star}). The value of $\epsilon$ at the transition is also a function of $y^{\star}$, $N$, and $n$:
\begin{align}
    \label{eps_y_star}
    \epsilon(y^{\star},N,n) = \frac{1}{2}y^{\star}-\frac{(n-1)}{2(N-n-1)}u(y^{\star}, N,n)\mathrm{tanh}(y^{\star}).
\end{align}
\end{proposition}

\begin{proof}
From previous discussions, the transition from a  supercritical bifurcation to a  subcritical bifurcation occurs when $\lambda_3$ crosses zero from negative to positive. The bifurcation corresponds to $\lambda_1=0$. So at the transition, the following equations should be satisfied:
\begin{numcases}{}
    g(y^{\star}, u, \epsilon, N,n) = 0 \label{condition_transition_1}\\
    \lambda_1(y^{\star}, u, \epsilon, N,n) = 0 \label{condition_transition_2}\\
    \lambda_3(y^{\star}, u, \epsilon, N,n) = 0 \label{condition_transition_3}.
\end{numcases}
Here $g(\cdot)$ denotes the LHS of Eqn. (\ref{y_star_eqn}). The dependence of $g$, $\lambda_1$ and $\lambda_3$ on variables and parameters is indicated.
Thus, given $N$ and $n$, which specify the network graph structure in the family of networks with three clusters, we can solve for $y^{\star}$, $u$ and $\epsilon$ from Eqn. (\ref{condition_transition_1})-(\ref{condition_transition_3}).

Eqns. (\ref{condition_transition_2})  and (\ref{condition_transition_3}) do not depend on $\epsilon$ explicitly.  Eqn. (\ref{condition_transition_2}) can be rearranged as the following quadratic equation:
\begin{align}
    c_2(y^{\star},N,n) u^2 + c_1(y^{\star},N,n) u + c_0(y^{\star},N,n) = 0
    \label{quadratic}
\end{align}
with $c_2$, $c_1$, $c_0$ given by (\ref{c2}), (\ref{c1}), (\ref{c0}).

Since $\mathrm{tanh}'(y^{\star}) \in (0,1]$, we get that $c_2>0$, $c_1>0$ and $c_0<0$. Thus, the quadratic equation  (\ref{quadratic}) has one positive and one negative solution. We are only interested in a positive $u$, so we can write $u$ as a function of $y^{\star}$, $N$ and $n$ as in (\ref{u_y_star}).

As $y^{\star}$ increases, $\mathrm{tanh}'(y^{\star})$ decreases. Then $c_1$ and $c_2$ decrease. Thus, the denominator of the RHS of Eqn. (\ref{u_y_star}) decreases. As the numerator is a positive constant, we see that $u(y^{\star},N,n)$ is a strictly increasing function of $y^{\star}$ with $u(0,N,n) = 1$ and $u(+\infty,N,n) = (N-1)/(N-2n-1)$.
From Eqn. (\ref{condition_transition_1}), we can express $\epsilon$ as a function of $y^\star$, $N$, and $n$ as given by (\ref{eps_y_star}).

In Eqns. (\ref{lambda1_origin}) and (\ref{lambda3_origin}), the terms in the big parenthesis in front of $u$ are the same. Thus, setting $\lambda_1=0$ in Eqn. (\ref{condition_transition_2}), we can simplify the expression for $\lambda_3$ to get 
\begin{align}
    \label{lambda3_proof}
    \lambda_3 = -\frac{1}{3}(N-1) + \frac{n(N-n-1)}{n-1} \Big(\frac{2a}{c^4}-\frac{4b^2}{c^5}\Big).
\end{align}
From (\ref{lambda3_proof}), we see that $\lambda_3$ depends on $N$, $n$, $a$, $b$ and $c$. From  (\ref{a})-(\ref{c}) and the fact that we can express $u$ as a function of $y^{\star}$, we can then express $\lambda_3$ as $\lambda_3(y^{\star},N,n)$ and get  (\ref{lambda3}).
\end{proof}

Our main theorem gives the condition for the existence of a transition from supercritical pitchfork to subcritical pitchfork in dynamics (\ref{reduce_3_1})-(\ref{reduce_3_3}). The existence only depends on the network structure, i.e., $N$ and $n$.   
\begin{theorem}
    \label{main_thm}
    Given $N$ and $n$, if there exists a $y^{\star}_+>0$ such that $\lambda_3(y^{\star}_+,N,n)>0$, then there exists $y^{\star}_0 \in (0,y^{\star}_+)$ and $y^{\star}_1 \in (y^{\star}_+,+\infty)$ such that $\lambda_3(y^{\star}_0,N,n)=\lambda_3(y^{\star}_1,N,n)=0$. In particular, the existence of $y^{\star}_0$ indicates a transition from supercritical pitchfork bifurcation to subcritical pitchfork bifurcation at $\epsilon(y^{\star}_0)$ and $u(y^{\star}_0)$. This implies a cascade in the network with three clusters. 
    If there does not exist such a $y^{\star}_+$, then there is no such transition and thus no cascade. 
\end{theorem}

\begin{proof}
    From the proof of Proposition \ref{prop_tran}, we know that $u(y^{\star})$ is a continuous function of $y^{\star}$ and $u(y^{\star})\in [1,(N-1)/(N-2n-1))$. Then from  (\ref{c}),  $c$ as a function of $y^{\star}$ does not blow up and is continuous in $y^{\star}$. Thus, from (\ref{c}), (\ref{u_y_star}), we have
    \begin{align}
        c(y^{\star}) &\!=\! \frac{n-1}{N-2n} \mathrm{tanh}'(y^{\star}) - \frac{N-n-1}{(N-2n)} 
        \frac{\sqrt{c_1^2-4c_2c_0}+c_1}{-2c_0} \nonumber \\
        &\leq \frac{n-1}{N-2n} \mathrm{tanh}'(y^{\star}) - \frac{N-n-1}{(N-2n)} 
        \frac{c_1+c_1}{-2c_0} \nonumber \\
        &= -\frac{(N-n-1)(N-2n-1)}{(N-2n)(N-1)}<0. \nonumber
    \end{align}
   Thus, from (\ref{lambda3}) it follows that $\lambda_3(y^{\star},N,n)$ does not blow up and is continuous in $y^{\star}$. Moreover, we have 
    \begin{align}
        \lambda_3(0,N,n) &= -\frac{1}{3}(N-1) - \frac{2}{3} \frac{n(N-n-1)}{(N-2n)}<0 \nonumber\\
        \lambda_3(\infty,N,n) &= -\frac{1}{3}(N-1)<0. \nonumber
    \end{align}
    If there exists a $y^{\star}_+>0$ such that $\lambda_3(y^{\star}_+,N,n)>0$, then from the continuity of $\lambda_3(y^{\star},N,n)$, we know there exists a $y^{\star}_0 \in (0,y^{\star}_+)$ and $y^{\star}_1 \in (y^{\star}_+,+\infty)$ such that $\lambda_3(y^{\star}_0,N,n)=\lambda_3(y^{\star}_1,N,n)=0$. Thus, $\lambda_3(y^{\star}_0,N,n)$ crosses zero from negative to positive, and from Proposition \ref{prop_tran},  there exists a transition from supercritical pitchfork bifurcation to subcritical pitchfork bifurcation in dynamics (\ref{reduce_3_1})-(\ref{reduce_3_3}). The value of $\epsilon$ and $u$ at which this transition happens can be calculated by $\epsilon(y^{\star}_0,N,n)$ and $u(y^{\star}_0,N,n)$ from Eqns. (\ref{eps_y_star}) and (\ref{u_y_star}), respectively.
\end{proof}

\begin{figure}[h]
    \subfloat{
        \includegraphics[width=0.22\textwidth]{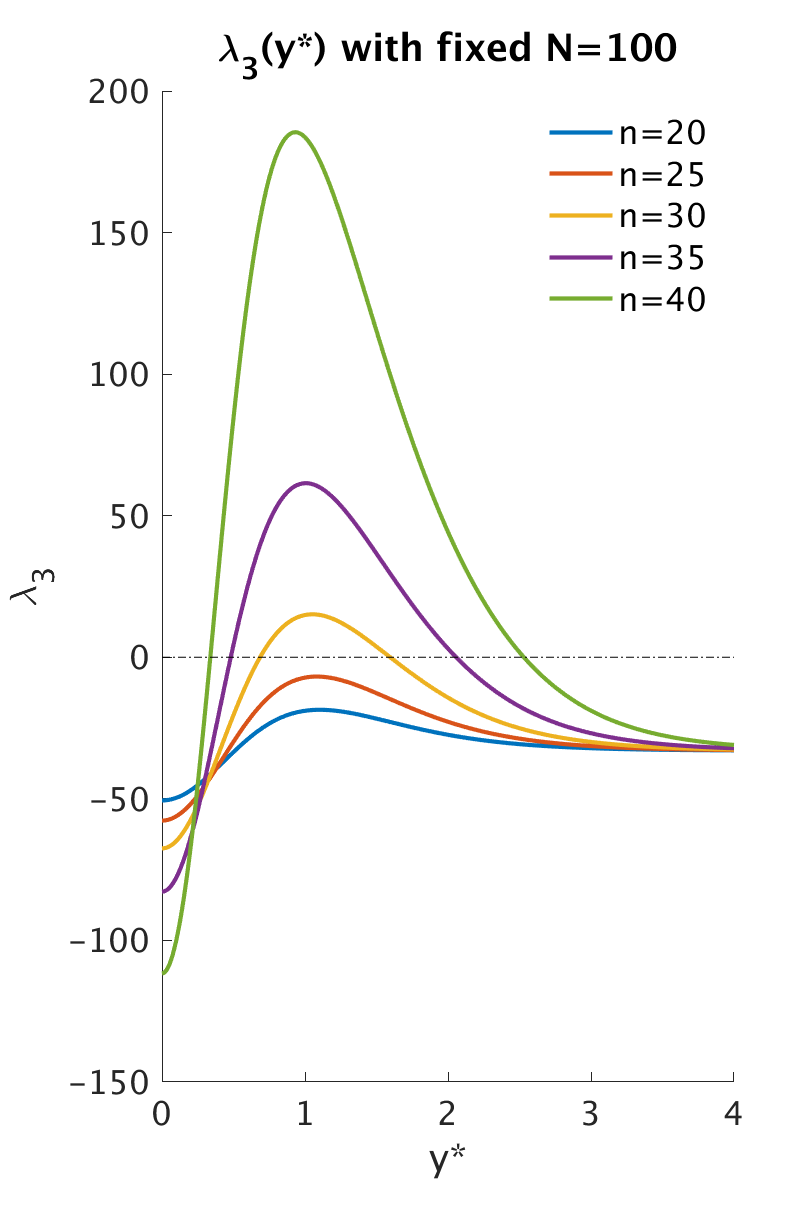}
    }
    \subfloat{
        \includegraphics[width=0.225\textwidth]{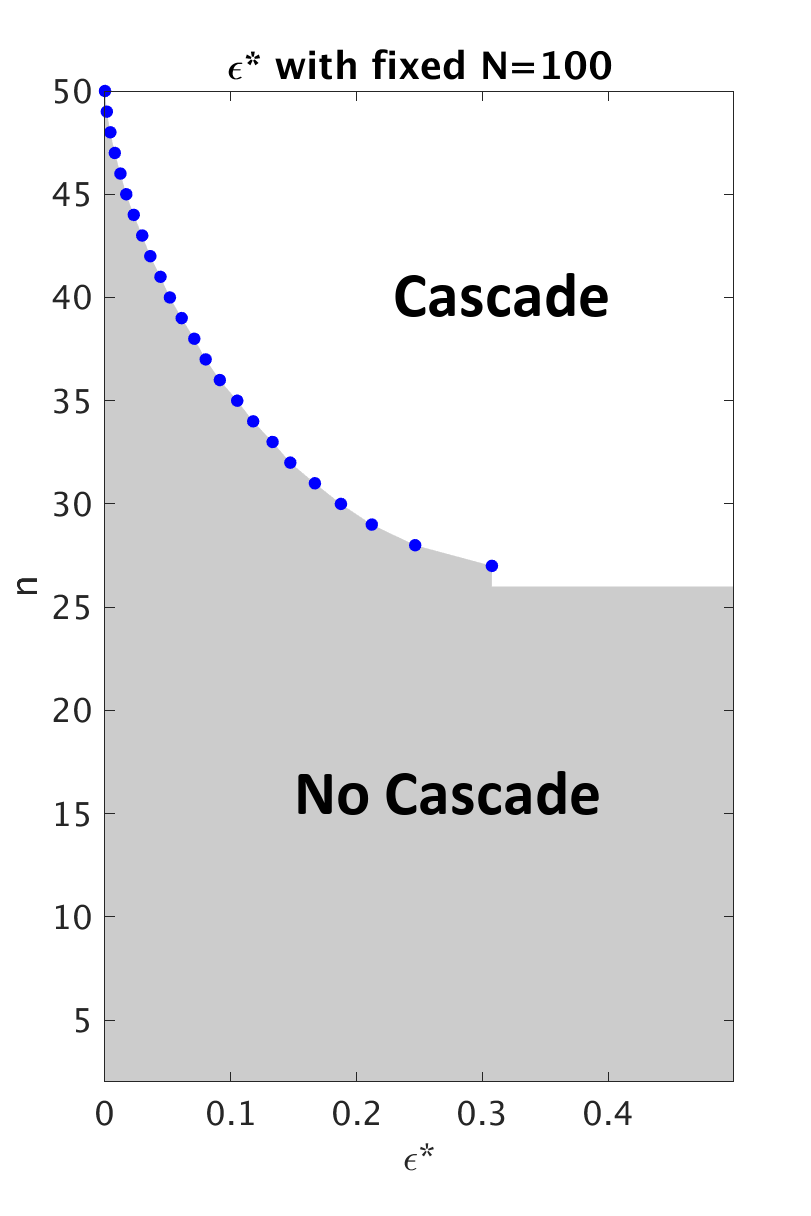}
    }
    \caption{The curves of $\lambda_3(y^{\star})$ for different values of $n$ and fixed $N$ (left). For lower $n$, $\lambda_3(y^{\star})$ remains negative. For higher $n$, $\lambda_3(y^{\star})=0$ has two solutions, and thus at the smaller solution $y^\star_0$, there is a transition from supercritical to subcritical pitchfork, and the possibility of a cascade. 
    Critical disparity $\epsilon^*$ for different values of $n$ and fixed $N$ (right). For  $n\geq 27$, $\epsilon > \epsilon^*$ leads to a cascade.
    }
    \label{N100}
\end{figure}

\begin{remark}
Fig.~\ref{N100} illustrates how the existence of a $y^{\star}_+>0$, and thus a cascade, depends on network structure parameters $N$, $n$, and $\epsilon$. For a fixed $N$, a large enough $n$, i.e., a large enough number of agents with disparity in thresholds,  is necessary for the cascade. For $n$ large enough that $y^{\star}_0$ exists, we can expect that for $\epsilon \in [0,\epsilon(y^{\star}_0,N,n)]$, the bifurcation is supercritical, since it is for $\epsilon=0$  \cite{gray2018multiagent}. As $\epsilon$ increases to greater than the critical value $\epsilon^{\star} = \epsilon(y_0^{\star},N,n)$, we expect to see the transition from no cascade to cascade.  For $N=100$, a cascade is possible if $n \geq 27$. The minimum disparity $\epsilon^{\star}$ that guarantees a cascade decreases as $n$ increases.  
\end{remark}

\section{An example}
\label{example}
We present a simulation of the CTM with the network structure shown in Fig. \ref{graph_11nodes} and $\epsilon=0.2$. The initial conditions of the 11 agents are picked randomly. Here, the average initial state is negative. 
We let $u_0 = 3$, $\kappa = 10$, and $\kappa_s=0.05$. Then  $u = 3 \mathrm{tanh}(10|\bar{x}_s|)$, where $\dot{\bar{x}}_s  = 0.05 ( \bar x- \bar{x}_s)$, $\bar x = \sum_{i=1}^{11} x_i/11$. 
Fig. \ref{sim_traj} shows how the states evolve. Agents in clusters 1, 2, and 3 are plotted in red, green, and blue, respectively. 
A perturbation $\beta=1$ is added to the dynamics (\ref{dynamics}) of an agent in the red cluster; its trajectory takes the largest value after the transient period. Except for the perturbed agent, states of all agents in each cluster quickly converge to a common value. 
So, we can interpret the results in terms of a perturbation of the reduced dynamics (\ref{reduce_3_1})-(\ref{reduce_3_3}). The solution converges to a perturbation of $\mathbf{y}^{\star} = [y^{\star},-y^{\star},0]^T$. Because of the perturbation, $\bar{x}_s$ slowly increases, which leads to a slow increase in $u$. At a certain time, $u$ crosses the bifurcation point, which leads to a cascade. 

\begin{figure}[h]
    \centering
    \includegraphics[width=0.48\textwidth]{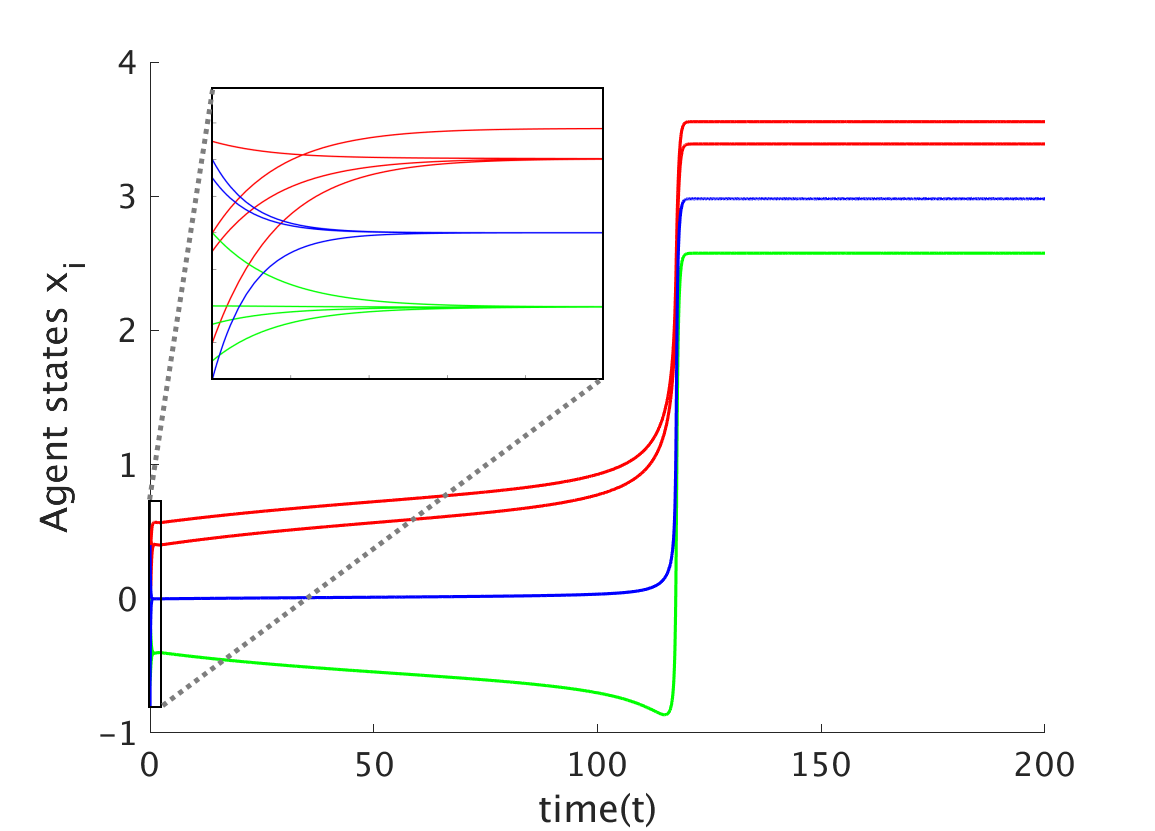}
    \caption{Agent state trajectories of the CTM in a network with three clusters, $N=11$, $n=4$. There is a cascade corresponding to the unfolded subcritical pitchfork as can be expected  since $0.2 = \epsilon>\epsilon(y_0^\star,N,n) = 0.11$.}
    \label{sim_traj}
\end{figure}

In this example, the graph structure $N=11$ and $n=4$ ensure the existence of $y^{\star}_0$ such that $\lambda_3(y^{\star}_0,N,n)=0$. Thus from Theorem \ref{main_thm}, there exists a transition from supercritical pitchfork to subcritical pitchfork in the symmetric system dynamics. Here $\epsilon(y^{\star}_0,N,n)= 0.11$. With a small $\epsilon$ (e.g., 0.1) the system exhibits a supercritical pitchfork; with a large epsilon (e.g., 0.2), the system exhibits a subcritical pitchfork. The introduction of an additive perturbation $\beta=1$ to the dynamics of a node in the high responsive group breaks the symmetry and lets the subcritical pitchfork unfold as shown in Fig. \ref{fig:perturbed} on the right. Thus, as we can see from the simulation, a cascade can be triggered even with a negative initial average state. 





\bibliographystyle{IEEEtran}
\bibliography{IEEEabrv,myref}

\end{document}